\documentclass[12pt,reqno]{amsart}

\usepackage[T1]{fontenc}
\usepackage[utf8]{inputenc}
\usepackage[english]{babel}
\usepackage{amsthm,amssymb,amsmath}

\usepackage[top=1in, bottom=1in, left=1in, right=1in]{geometry}
\usepackage[colorlinks=true,linkcolor=blue,citecolor=blue,urlcolor=blue]{hyperref}

\newtheorem{theorem}{Theorem}[section]
\newtheorem{proposition}[theorem]{Proposition}
\newtheorem{corollary}[theorem]{Corollary}
\newtheorem{lemma}[theorem]{Lemma}

\theoremstyle{definition}

\DeclareMathOperator{\Hol}{Hol}
\DeclareMathOperator{\Aut}{Aut}
\DeclareMathOperator{\Soc}{Soc}
\DeclareMathOperator{\PSL}{PSL}
\DeclareMathOperator{\PSU}{PSU}
\DeclareMathOperator{\pr}{pr}
\newcommand{\KT}{\mathsf K}
\newcommand{\id}{\mathrm{id}}

\title[Kernel--wreath constructions]
{Kernel--wreath constructions and infinite families of finite simple skew braces}

\author[Marco Damele]{Marco Damele}
\thanks{Dipartimento di Matematica, Università di Cagliari, Via Ospedale 72, 09124 Cagliari, Italy; \texttt{marco.damele@unica.it}; ORCID 0009-0008-3088-5766}


\begin{document}

\begin{abstract}
We introduce a kernel--wreath construction for finite skew braces. Let
$C$ be a finite skew brace, write $(C,+)=A$ and $(C,\circ)=R$, and let
$\chi:R\twoheadrightarrow T$ be an epimorphism onto a non-trivial finite
abelian group. We show that a natural index-$|T|$ kernel in the
permutational wreath product $R\wr T$ acts regularly on $A^T$, and hence
defines a new skew brace $\KT_T(C,\chi)$ with additive group
$A^{|T|}$. The construction preserves every finite abelian quotient of
the multiplicative group, as well as solvability, and contains the seed
brace $C$ as a diagonal subbrace. It can therefore be iterated
indefinitely. More precisely, if $Q$ is a non-trivial finite abelian
quotient of $R$ and $p\in\pi(Q)$, then one obtains an infinite tower
\[
C=C_0\hookrightarrow C_1\hookrightarrow C_2\hookrightarrow\cdots
\]
with $(C_m,+)\cong A^{p^m}$ for every $m\geq0$.
Our main permanence result shows that if $C$ is simple and is not a
trivial skew brace, then $\KT_T(C,\chi)$ is again simple. Consequently,
a single simple seed whose multiplicative group has a non-trivial finite
abelian quotient gives rise to infinite families of finite simple skew
braces. In particular, starting from suitable solvable regular subgroups
of the holomorph of a finite non-abelian simple group $S$, we obtain
infinite families of simple skew braces with additive groups
$S^{p^m}$ and solvable multiplicative groups. Further applications are
given to the simple skew braces of order $12$ and to Byott's family.
\end{abstract}

\maketitle

\section{Introduction}

Skew braces, introduced by Guarnieri and Vendramin \cite{GV17}, provide
an algebraic description of regular subgroups of holomorphs and play a
central role in the study of non-degenerate set-theoretic solutions of
the Yang--Baxter equation. A skew brace consists of two group structures
on the same underlying set, related by a compatibility condition, and
its ideals are the subgroups which are simultaneously compatible with
the additive structure, the multiplicative structure and the lambda
action. The construction and classification of finite simple skew braces
are therefore natural structural problems.

Several constructions of finite simple braces are known when the
additive group is abelian. The first non-trivial finite simple braces
were obtained by Bachiller using matched products \cite{Bachiller18},
and further constructions have produced large families of simple braces
\cite{CJO21}. The genuinely skew case is less developed. The
smallest simple skew braces which are not braces have order $12$; there
are two such examples, with additive group $A_4$ and multiplicative
group $C_3\rtimes C_4$ \cite{KSV21}. More recently, Byott constructed
an infinite family of simple skew braces of order $p^pq$, for suitable
primes $p$ and $q$, with both additive and multiplicative groups
solvable \cite{Byott26}.

The aim of this paper is to introduce a general construction that
starts from a finite skew brace and produces new skew braces on direct
powers of its additive group. The construction has three features that
make it particularly useful: it is explicit, it can be iterated, and,
under a natural hypothesis on the seed, it preserves simplicity. Let $C$ be a finite skew brace and write $(C,+)=A$ and $(C,\circ)=R$.
Assume that $R$ admits an epimorphism
$\chi:R\twoheadrightarrow T$ onto a non-trivial finite abelian group.
Let $\Omega$ be the underlying set of $T$. We consider the permutational
wreath product $W=R^\Omega\rtimes T$
and define
\[
\varepsilon:W\longrightarrow T,
\qquad
\varepsilon((r_x)_{x\in\Omega},u)
=
\left(\prod_{x\in\Omega}\chi(r_x)\right)u.
\]
The first main result shows that the kernel of this map has precisely the
right size and action to define a new skew brace.

\medskip
\noindent\textbf{Theorem A.}
\emph{The map $\varepsilon$ is a surjective homomorphism and
$G=\ker\varepsilon$ is a regular subgroup of $\Hol(A^\Omega)$. Hence
$G$ determines a finite skew brace, denoted by
$\KT_T(C,\chi)$, satisfying
\[
(\KT_T(C,\chi),+)\cong A^{|T|}
\qquad\text{and}\qquad
(\KT_T(C,\chi),\circ)\cong G.
\]
}

The construction may be viewed as a balanced version of the natural
wreath-product action. The base group $R^\Omega$ allows the original
brace to act independently on each coordinate, while the top group $T$
permutes the coordinate factors transitively. The single relation
$\varepsilon=1$ reduces the order of the wreath product by precisely the
factor $|T|$, required for regularity, without destroying
either the local action of $R$ or the permutation of the coordinates.
A crucial feature of the construction is that finite abelian quotients
of the multiplicative group survive. More precisely, if
$\theta:R\twoheadrightarrow Q$ is an epimorphism onto a finite abelian
group, then
\[
((r_x)_{x\in\Omega},u)
\longmapsto
\prod_{x\in\Omega}\theta(r_x)
\]
defines an epimorphism from the multiplicative group of
$\KT_T(C,\chi)$ onto $Q$. Thus the same quotient can be used again at
the next step. Moreover, the seed brace occurs canonically as a diagonal
subbrace of the new one. These two facts yield infinite towers.

\medskip
\noindent\textbf{Theorem B.}
\emph{Suppose that $(C,\circ)=R$ admits an epimorphism onto a non-trivial
finite abelian group $Q$. For every prime $p\in\pi(Q)$ there exists an
infinite sequence of finite skew braces
$C=C_0\hookrightarrow C_1\hookrightarrow C_2\hookrightarrow\cdots$
such that
\[
(C_m,+)\cong A^{p^m}
\]
for every $m\geq0$. Each $C_m$ embeds diagonally as a subbrace of
$C_{m+1}$, and the multiplicative group of every $C_m$ still admits
$Q$ as a quotient. If $R$ is solvable, then every $(C_m,\circ)$ is
solvable.}

More generally, if $n\geq2$ satisfies
$\pi(n)\subseteq\pi(Q)$, the construction can be iterated using the
prime factors of $n$ and yields a finite skew brace $B_n$ with
$(B_n,+)\cong A^n$. Thus the kernel--wreath construction is not merely
a one-step construction, but a mechanism for producing skew braces on
arbitrarily large direct powers of a fixed additive group.

The main permanence theorem concerns simplicity.

\medskip
\noindent\textbf{Theorem C.}
\emph{Let $C$ be a finite simple skew brace which is not a trivial skew
brace, and let $\chi:(C,\circ)\twoheadrightarrow T$ be an epimorphism
onto a non-trivial finite abelian group. Then
$\KT_T(C,\chi)$ is simple.}

The proof exploits the internal structure supplied by the kernel. The
lambda action is transitive on the coordinate factors of $A^\Omega$,
while compensated base elements recover the full lambda action of $C$
on any chosen coordinate. At the same time, diagonal elements reproduce
the multiplicative conjugation of the seed brace. If $I$ is a proper
ideal of $\KT_T(C,\chi)$, these elements allow us to show that both its
intersections with the coordinate factors and its coordinate
projections give ideals of $C$. The simplicity of $C$ then forces a
strong restriction on $I$, and the remaining possibility is excluded
using the rigidity result for finite simple skew braces with nilpotent
multiplicative group from \cite[Corollary~3.2]{DameleErcan26}.

Combining Theorems~B and~C yields one of the main consequences of the
paper: a single simple seed with a non-trivial finite abelian quotient
of its multiplicative group produces infinitely many finite simple skew
braces. More precisely, if $Q$ is such a quotient and
$p\in\pi(Q)$, then one obtains a nested tower
$C=C_0\hookrightarrow C_1\hookrightarrow C_2\hookrightarrow\cdots$
of finite simple skew braces satisfying
$(C_m,+)\cong (C,+)^{p^m}$
for every $m\geq0$. If $(C,\circ)$ is solvable, then all the
multiplicative groups occurring in the tower remain solvable.
In particular, every finite simple skew brace which is not trivial and
has solvable multiplicative group gives rise to an infinite family of
finite simple skew braces with solvable multiplicative groups. Indeed,
a non-trivial finite solvable group has non-trivial abelianization.
The construction is especially effective when the additive group of the
seed is a finite non-abelian simple group. Let $S$ be such a group and
let $R\leq\Hol(S)$ be a solvable regular subgroup. The corresponding
skew brace is automatically simple, and for every
$p\in\pi(R/R')$ our construction gives the following family.

\medskip
\noindent\textbf{Corollary D.}
\emph{Let $S$ be a finite non-abelian simple group and let
$R\leq\Hol(S)$ be a solvable regular subgroup. For every
$p\in\pi(R/R')$ there exists an infinite sequence of finite simple skew
braces $B_0,B_1,B_2,\ldots$ such that
$(B_m,+)\cong S^{p^m}$
for every $m\geq0$, while every $(B_m,\circ)$ is solvable. Moreover, the
braces may be chosen to form a nested tower under diagonal subbrace
embeddings.}

Thus the construction produces systematically simple skew brace
structures on arbitrarily large direct powers of a fixed non-abelian
simple group. Notice that, for $m>0$, the additive group $S^{p^m}$ is
characteristically simple but is no longer simple as a group; its
simplicity as a skew brace is therefore a genuinely new phenomenon
coming from the interaction between the coordinate factors.

Tsang classified the finite non-abelian simple groups whose holomorph
contains a solvable regular subgroup \cite{Tsang23}. They are
$\PSL_2(q)$, with $q\neq2,3$ a prime power, together with
\[
\PSL_3(3),\quad
\PSL_3(4),\quad
\PSL_3(8),\quad
\PSU_3(8),\quad
\PSU_4(2),\quad
M_{11}.
\]
Consequently, every group in this list gives rise to infinite families
of simple skew braces on direct powers of the same simple group. Two
examples illustrate the construction explicitly. For $S=A_5$, the exact factorization $A_5=A_4C_5$ gives a solvable
regular subgroup $R\cong A_4\times C_5$
of $\Hol(A_5)$. Since $R/R'\cong C_3\times C_5\cong C_{15},$
both primes $3$ and $5$ may be used in the construction.

\medskip
\noindent\textbf{Corollary E.}
\emph{There exist two infinite families of finite simple skew braces,
denoted by $B_m^{(3)}$ and $B_m^{(5)}$ for $m\geq0$, such that
\[
(B_m^{(3)},+)\cong A_5^{3^m},
\qquad
(B_m^{(5)},+)\cong A_5^{5^m},
\]
and all their multiplicative groups are solvable. More generally, for
every $n\geq2$ with $\pi(n)\subseteq\{3,5\}$ there exists a finite
simple skew brace $B_n$ such that $(B_n,+)\cong A_5^n$
and $(B_n,\circ)$ is solvable.}

A second example is provided by $\PSU_3(8)$. The solvable regular
subgroup constructed by Tsang in
\cite[Proof of Theorem~1.3, case~(3)]{Tsang23} admits a quotient
isomorphic to $C_3$. Hence our construction yields another explicit
tower.

\medskip
\noindent\textbf{Corollary F.}
\emph{For every $m\geq0$ there exists a finite simple skew brace $B_m$
such that $(B_m,+)\cong\PSU_3(8)^{3^m},$
while $(B_m,\circ)$ is solvable.}

The same mechanism applies well beyond simple additive groups. Starting
from either of the two simple skew braces of order $12$, whose additive
group is $A_4$ and whose multiplicative group admits a quotient $C_2$,
we obtain simple skew braces with additive groups
\[
A_4,\quad A_4^2,\quad A_4^4,\quad A_4^8,\ldots
\]
and solvable multiplicative groups. Similarly, every skew brace in
Byott's family provides a seed for an infinite tower: if $N$ denotes
its additive group and $p$ is the prime occurring in the construction,
then we obtain finite simple skew braces with additive groups
\[
N,\quad N^p,\quad N^{p^2},\quad N^{p^3},\ldots
\]
and solvable multiplicative groups.
These examples show that the kernel--wreath construction does not merely
produce isolated new simple skew braces. Rather, it turns a single
suitable seed into an infinite and explicitly controlled family, while
preserving both simplicity and important structural properties of the
multiplicative group.

The paper is organized as follows. Section~\ref{sec:prelim} recalls the
basic facts on skew braces and regular subgroups of holomorphs. In
Section~\ref{sec:construction} we introduce the kernel--wreath
construction, prove Theorem~A, establish preservation of finite abelian
quotients and solvability, and construct the diagonal copy of the seed.
The same section also develops the iteration procedure and the infinite
towers of Theorem~B. Section~\ref{sec:simplicity} is devoted to the
proof of Theorem~C. Finally, Section~\ref{sec:applications} develops the
families described above, with particular attention to non-abelian
simple additive groups, the simple skew braces of order $12$, and
Byott's family.

\section{Preliminaries}\label{sec:prelim}

\subsection{Skew braces}

We follow the terminology and basic notation on skew braces from
\cite{GV17,KSV21}. A \emph{skew brace} is a triple $(B,+,\circ)$ such that $(B,+)$ and $(B,\circ)$ are groups and $a\circ(b+c)=a\circ b-a+a\circ c$
for all $a,b,c\in B$. The identity elements of the two groups coincide. The map $\lambda^{B}:(B,\circ)\longrightarrow \Aut(B,+)$ given by $\lambda^{B}(a)(b)=\lambda^{B}_a(b)=-a+a\circ b,$
is a group homomorphism. We write $a*b=\lambda^{B}_a(b)-b.$
The skew brace is \emph{trivial} if $\lambda^{B}_a=\id$ for every $a\in B$, equivalently if $a\circ b=a+b$ for all $a,b\in B$. A subgroup $I\leq(B,+)$ is an \emph{ideal} if $I\trianglelefteq(B,+),$ $I\trianglelefteq(B,\circ)$ and for all $a \in B$ $\lambda^{B}_{a}(I)=I$.
A non-zero skew brace is \emph{simple} if its only ideals are $0$ and $B$. The socle is $\Soc(B)=Z(B,+)\cap\ker\lambda^{B}.$
It is an ideal of $B$. Hence, if $B$ is simple and is not a trivial skew brace, then $\Soc(B)=0.$  

\subsection{Holomorph and skew braces}

Let \(H\) be a finite group. Recall that its holomorph is $\mathrm{Hol}(H)=H\rtimes \mathrm{Aut}(H)$ where $\mathrm{Aut}(H)$ acts on $H$ by evaluation. Every element of \(\mathrm{Hol}(H)\) can be written as a pair \((x,f)\), with $x \in H$ and $f \in \mathrm{Aut}(H)$ and \(\mathrm{Hol}(H)\) acts naturally on \(H\) via $(x,f)\cdot h=x\,f(h)$ for all $h \in H.$
Now let \(G\leq \mathrm{Hol}(H)\). By restriction, \(G\) acts on \(H\) through the same rule. We say that \(G\) is a \emph{regular subgroup} of \(\mathrm{Hol}(H)\) if this action is regular. Recall that the action of $G$ on $H$ is regular if it is transitive and free
(equivalently, $\operatorname{Stab}_G(0)=1$ and $G\cdot 0=H$). Equivalently, \(G\) is regular if for every \(h_1,h_2\in H\) there exists a unique \(g\in G\) such that $g\cdot h_1=h_2.$
The connection between skew braces and regular subgroups of the holomorph is summarized in the following well-known theorem.

\begin{theorem}[{\cite[Theorem~4.2]{GV17}}] \label{Connection}
Let \(H\) be a group and let \(G\) be a regular subgroup of \(\mathrm{Hol}(H)\). Then there exists a skew brace \((B,+,\cdot)\) such that $(B,+)\cong H$ and $(B,\cdot)\cong G$. Conversely, if \((B,+,\cdot)\) is a skew brace, then \((B,\cdot)\) can be identified with a regular subgroup of \(\mathrm{Hol}(B,+)\) via the map $(B,\cdot)\longrightarrow \mathrm{Hol}(B,+),
 \ b\longmapsto (b,\lambda^{B}_b)$.
\end{theorem}

\section{The kernel--wreath construction}\label{sec:construction}

\subsection{The new construction}

Let $C$ be a finite skew brace and write $(C,+)=A$ and $(C,\circ)=R$. Write $\lambda^C:R\to\Aut(A)$ for its lambda action. Since
$A=(C,+)$ and $R=(C,\circ)$ have the same underlying set, let $\pi:R\longrightarrow A$
denote the identity map on the underlying set of $C$. Thus, for every
$r,s\in R$ we have $\pi(rs)=\pi(r)+\lambda^C_r(\pi(s)).$
Hence $\pi$ is a bijective crossed homomorphism with respect to the action
$\lambda^C$. Let $T$ be a non-trivial finite abelian group and let $\chi:R\twoheadrightarrow T$
be an epimorphism.
Let $\Omega$ be the underlying set of $T$.
Observe that $T$ acts on the direct product $R^\Omega=\prod_{x\in\Omega}R$
by permutation of the coordinates. Explicitly, if $r=(r_x)_{x\in\Omega}\in R^\Omega$
and $u\in T$, we define ${}^u r=(r_{u^{-1}x})_{x\in\Omega}.$
We then consider the corresponding permutational wreath product $W=R\wr_\Omega T
   =R^\Omega\rtimes T.$
The natural transitive action of $W$ on $A^\Omega$ is given by
\begin{equation}\label{eq:wreath-action}
(r,u)\cdot(a_x)_{x\in\Omega}
 =\bigl(\pi(r_x)+\lambda^C_{r_x}(a_{u^{-1}x})\bigr)_{x\in\Omega}.
\end{equation}
For $(r,u)\in W=R^\Omega\rtimes T$, define
$\alpha_{(r,u)}((a_x)_{x\in\Omega})
=
\bigl(\lambda^C_{r_x}(a_{u^{-1}x})\bigr)_{x\in\Omega}.$
Then
\[
\Phi:W\longrightarrow \Hol(A^\Omega),
\qquad
\Phi((r_x)_{x\in\Omega},u)
=
\left(
(\pi(r_x))_{x\in\Omega},
\alpha_{(r,u)}
\right)
\]
is an injective group homomorphism. Indeed, the crossed-homomorphism identity
for $\pi$ and the fact that $\lambda^C$ is a homomorphism show that $\Phi$
respects multiplication in $W$. If $\Phi((r_x)_{x\in\Omega},u)$ is the
identity, then $\pi(r_x)=0$ for every $x$, hence $r_x=1_R$ for every $x$.
The remaining automorphism is the permutation of the coordinates induced by
$u$. Since $T$ is a non-trivial quotient of $R$, the group $A$ is non-trivial;
hence the coordinate-permutation action of $T$ on $A^\Omega$ is faithful.
Thus $u=1_T$, and we can view $W$ as a subgroup of $\Hol(A^\Omega)$.
Define
\begin{equation}\label{eq:epsilon}
\varepsilon:W\longrightarrow T,
\qquad
\varepsilon((r_x)_{x\in\Omega},u)
 =\left(\prod_{x\in\Omega}\chi(r_x)\right)u.
\end{equation}

\begin{theorem}[Kernel--wreath construction]\label{thm:construction}
The map $\varepsilon$ is a surjective homomorphism. Its kernel $G=\ker\varepsilon$
is a regular subgroup of $\Hol(A^\Omega)$. Consequently $G$ determines a skew brace, denoted $\KT_T(C,\chi),$
with $(\KT_T(C,\chi),+)\cong A^\Omega\cong A^{|T|}$ and $(\KT_T(C,\chi),\circ)\cong G.$ 
\end{theorem}

\begin{proof}
Since $T$ is abelian and the top group only permutes coordinates,
\begin{align*}
\varepsilon((r,u)(s,v))
&=\varepsilon(r\,{}^u s,uv)=\left(\prod_{x\in\Omega}\chi(r_xs_{u^{-1}x})\right)uv=\left(\prod_{x\in\Omega}\chi(r_x)\right)
  \left(\prod_{x\in\Omega}\chi(s_{u^{-1}x})\right)uv\\
&=\varepsilon(r,u)\varepsilon(s,v).
\end{align*}
Thus $\varepsilon$ is a homomorphism. It is surjective since $\varepsilon((1)_{x\in\Omega},u)=u.$
Therefore $|G|=|W|/|T|=|R|^{|T|}=|A|^{|T|}=|A^\Omega|.$
Let $0=(0)_{x\in\Omega}\in A^\Omega$, and identify the top group with
\[
\widehat T=\{((1_R)_{x\in\Omega},u):u\in T\}\leq W.
\]
By \eqref{eq:wreath-action}, the stabilizer of $0$ in $W$ is precisely
$\widehat T$. Moreover $G\cap\widehat T=1$, since
$\varepsilon((1_R)_{x\in\Omega},u)=u$. It follows that the stabilizer
of $0$ in $G$ is trivial. In other words, $\operatorname{Stab}_G(0)=1.$
By the orbit--stabilizer theorem we have $|G\cdot 0|
=
[G:\operatorname{Stab}_G(0)]
=
|G|=|A^\Omega|.$
It follows that $G\cdot 0=A^\Omega.$
Thus the action of $G$ on $A^\Omega$ is transitive. Since $\operatorname{Stab}_G(0)=1$, the action is regular. The conclusion now follows from Theorem~\ref{Connection}.
\end{proof}

\subsection{Properties of the multiplicative group}

We record the elements of $G$ that will be used repeatedly.

\begin{lemma} \label{lem:elements}
Let $B=\KT_T(C,\chi)$
be the kernel--wreath skew brace, with $(B,+)=A^\Omega$ and $(B,\circ)\cong G=\ker\varepsilon\leq W=R^\Omega\rtimes T.$
Let $\omega,\nu\in\Omega$ be distinct. Then the following hold.

\begin{enumerate}

\item
Let $r,s\in R$ satisfy $\chi(r)\chi(s)=1_T.$ Define $b=(b_x)_{x\in\Omega}\in R^\Omega$ by $b_\omega=r$, $b_\nu=s$ and $ b_x=1_R $ for all $x\notin\{\omega,\nu\}$. Then $(b,1_T)\in G.$ \item For every $r\in R$, the diagonal element $d_r=\bigl((r)_{x\in\Omega},1_T\bigr)\in W$ belongs to $G$. 

\item Let $\rho:W=R^\Omega\rtimes T\longrightarrow T$ be the canonical projection onto the top group. Then the restriction $\rho|_G:G\longrightarrow T$ is surjective. For each $x\in\Omega$, let $A_x = \{(a_y)_{y\in\Omega}\in A^\Omega : a_y=0 \text{ for all } y\neq x\}$. The lambda action of an element $g=((r_y)_{y\in\Omega},u)\in G$ satisfies $\lambda^B_g(A_x)=A_{ux}.$ Since the projection $G\to T$ is surjective and $T$ acts transitively on $\Omega$, the lambda action of $G$ is transitive on the set of coordinate factors $ \{A_x:x\in\Omega\}.$
\end{enumerate}
\end{lemma}

\begin{proof}
For (1), by the definition of $b$ and of $\varepsilon$, we have $\varepsilon(b,1_T)=
\chi(r)\chi(s)=
1_T.$
Thus $(b,1_T)\in\ker\varepsilon=G$. For (2), let $r\in R$. Since $\chi(r)\in T$, we get
\begin{align*}
\varepsilon(d_r)
=
\left(\prod_{x\in\Omega}\chi(r)\right)1_T
=
\chi(r)^{|\Omega|}
=
\chi(r)^{|T|}
=
1_T
\end{align*}
where the last equality follows from Lagrange's theorem. Thus $d_r\in\ker\varepsilon=G$. For (3), let $u\in T$. Since $\chi:R\twoheadrightarrow T$ is surjective,
there exists $r\in R$ such that $\chi(r)=u^{-1}.$
Fix $\omega\in\Omega$, and define $b=(b_x)_{x\in\Omega}\in R^\Omega$ by $b_\omega=r$ and $b_x=1_R$ for all $x\neq\omega.$
Then
\begin{align*}
\varepsilon(b,u)
&=
\left(\prod_{x\in\Omega}\chi(b_x)\right)u=
\chi(r)u=
u^{-1}u=
1_T.
\end{align*}
Hence $(b,u)\in G$. Since $\rho(b,u)=u,$
and $u\in T$ was arbitrary, the restriction
$\rho|_G:G\longrightarrow T$
is surjective.
Let $g=((r_y)_{y\in\Omega},u)\in G.$
Its lambda action on $A^\Omega$ is given by
\[
\lambda^B_g((a_y)_{y\in\Omega})
=
\bigl(
\lambda^C_{r_y}(a_{u^{-1}y})
\bigr)_{y\in\Omega}.
\]
If $a=(a_y)_{y\in\Omega}\in A_x$, then $a_z=0$ for every $z\neq x$.
Hence the $y$-th coordinate of $\lambda^B_g(a)$ can be non-zero only if $u^{-1}y=x,$
that is, only if $y=ux.$
Therefore $\lambda^B_g(A_x)\subseteq A_{ux}.$
Since $\lambda^B_g$ is an automorphism of $A^\Omega$, we actually have $\lambda^B_g(A_x)=A_{ux}.$ Thus the permutation induced by $g$ on the set of coordinate factors
depends only on its top component $u$ and is given by $A_x\longmapsto A_{ux}.$
Now let $x,z\in\Omega$. Since $\Omega=T$ and $T$ acts on itself by left
multiplication, there exists $u\in T$ such that $ux=z.$
By the surjectivity of $\rho|_G:G\twoheadrightarrow T,$
there exists $g\in G$ whose top component is $u$. For this element, $\lambda^B_g(A_x)=A_z.$
Hence the lambda action of $G$ is transitive on the set of coordinate factors $\{A_x:x\in\Omega\}.$
\end{proof}

\begin{proposition}\label{prop:solvable}
Let $B=\KT_T(C,\chi)$
be the kernel--wreath skew brace, with $(B,+)=A^\Omega$ and $(B,\circ)\cong G=\ker\varepsilon\leq W=R^\Omega\rtimes T.$ If $R$ is solvable, then $G$ is solvable.
\end{proposition}

\begin{proof}
The group $R^\Omega\rtimes T$ is solvable because $R$ is solvable and $T$ is abelian. Hence its subgroup $G$ is solvable.
\end{proof}

The following observation is the key to the iteration procedure.

\begin{proposition}[Preservation of abelian quotients]\label{prop:quotients}
Let $B=\KT_T(C,\chi)$
be the kernel--wreath skew brace, with $(B,+)=A^\Omega$ and $(B,\circ)\cong G=\ker\varepsilon\leq W=R^\Omega\rtimes T.$ Let $Q$ be a finite abelian group and let
$\theta:R\twoheadrightarrow Q$
be an epimorphism. Then
\[
\Theta:G\longrightarrow Q,
\qquad
\Theta((r_x)_{x\in\Omega},u)=\prod_{x\in\Omega}\theta(r_x),
\]
is an epimorphism.
\end{proposition}

\begin{proof}
Let $g=((r_x)_{x\in\Omega},u)$ and $h=((s_x)_{x\in\Omega},v)$
be elements of $G$. Recall that multiplication in the wreath product is $gh
=
\left(
(r_xs_{u^{-1}x})_{x\in\Omega},
uv
\right).$
Therefore
\begin{align*}
\Theta(gh)
&=
\prod_{x\in\Omega}
\theta(r_xs_{u^{-1}x})=
\prod_{x\in\Omega}
\theta(r_x)\theta(s_{u^{-1}x})=
\left(\prod_{x\in\Omega}\theta(r_x)\right)
\left(\prod_{x\in\Omega}\theta(s_{u^{-1}x})\right).
\end{align*}
Since the map $x\longmapsto u^{-1}x$
is a permutation of $\Omega$, we have
\[
\prod_{x\in\Omega}\theta(s_{u^{-1}x})
=
\prod_{x\in\Omega}\theta(s_x).
\]
Hence
\[
\Theta(gh)
=
\left(\prod_{x\in\Omega}\theta(r_x)\right)
\left(\prod_{x\in\Omega}\theta(s_x)\right)
=
\Theta(g)\Theta(h).
\]
Here we use that $Q$ is abelian, so the factors may be reordered.
Thus $\Theta:G\to Q$ is a group homomorphism.
We now prove that $\Theta$ is surjective. Let $q\in Q$. Since $\theta:R\twoheadrightarrow Q$
is surjective, choose $r\in R$ such that $\theta(r)=q.$
Fix $\omega\in\Omega$ and define $b=(b_x)_{x\in\Omega}\in R^\Omega$
by $b_\omega=r$ and $b_x=1_R$ for all $x\neq\omega.$
Set $u=\chi(r)^{-1}\in T.$
Then
\begin{align*}
\varepsilon(b,u)
&=
\left(\prod_{x\in\Omega}\chi(b_x)\right)u=
\chi(r)\chi(r)^{-1}=
1_T.
\end{align*}
Hence $(b,u)\in\ker\varepsilon=G.$
Finally,
\begin{align*}
\Theta(b,u)
&=
\prod_{x\in\Omega}\theta(b_x)=
\theta(r)=
q.
\end{align*}
Since $q\in Q$ was arbitrary, $\Theta$ is surjective.
\end{proof}

For a finite group $H$, we denote by $\pi(H)$ the set of prime divisors of $|H|$.
\begin{theorem}[Iteration theorem]\label{thm:iteration}
Let $C$ be a finite skew brace and write $(C,+)=A$ and $(C,\circ)=R$.
Suppose that $R$ admits an epimorphism onto a non-trivial finite abelian
group $Q$. Then, for every integer $n\geq 2$ satisfying
$\pi(n)\subseteq\pi(Q)$, there exists a finite skew brace $B_n$ such that
$(B_n,+)\cong A^n$. Moreover, $(B_n,\circ)$ admits an epimorphism onto $Q$. If $R$ is
solvable, then $(B_n,\circ)$ is solvable.
\end{theorem}

\begin{proof}
Write $n=p_1p_2\cdots p_m$, where the primes $p_i$, not necessarily
distinct, belong to $\pi(Q)$. Let $\theta_0:R\twoheadrightarrow Q$ be an
epimorphism. Since $Q$ is finite abelian, for every $i$ there exists an
epimorphism $\eta_i:Q\twoheadrightarrow C_{p_i}$. Set $C_0=C$ and $R_0=R$. Suppose that $C_{i-1}$ has been constructed,
with multiplicative group $R_{i-1}=(C_{i-1},\circ)$ and an epimorphism
$\theta_{i-1}:R_{i-1}\twoheadrightarrow Q$. Define $\chi_i=\eta_i\theta_{i-1}:R_{i-1}\twoheadrightarrow C_{p_i}$
and put $C_i=\KT_{C_{p_i}}(C_{i-1},\chi_i).$
By Proposition~\ref{prop:quotients}, the multiplicative group
$R_i=(C_i,\circ)$ again admits an epimorphism
$\theta_i:R_i\twoheadrightarrow Q$. Hence the construction can be
iterated. At each step, $(C_i,+)\cong(C_{i-1},+)^{p_i},$
and therefore $(C_m,+)\cong A^{p_1\cdots p_m}=A^n$. Setting $B_n=C_m$
gives the desired skew brace.
Finally, if $R$ is solvable, then solvability of the multiplicative group
is preserved at every step by Proposition~\ref{prop:solvable}.
\end{proof}

\subsection{Viewing \texorpdfstring{$C$}{C} inside \texorpdfstring{$\KT_T(C,\chi)$}{KT(C,chi)}}

The construction also contains a canonical copy of the seed brace.

\begin{proposition}[Diagonal subbrace]\label{prop:diagonal}
The diagonal subgroup $\Delta(A)=\{(a)_{x\in\Omega}:a\in A\}\leq A^\Omega$
is a subbrace of $\KT_T(C,\chi)$, and $C\longrightarrow\Delta(A), a\longmapsto(a)_{x\in\Omega}$
is an isomorphism of skew braces.
\end{proposition}

\begin{proof}
By Lemma~\ref{lem:elements}, for every $r\in R$ the diagonal element $d_r=\bigl((r)_{x\in\Omega},1_T\bigr)$
belongs to $G$. Consider the map $\delta:R\longrightarrow G, \delta(r)=d_r.$
Since the top component of every $d_r$ is trivial, multiplication is
coordinatewise, and therefore $d_rd_s
=
\bigl((rs)_{x\in\Omega},1_T\bigr)
=
d_{rs}$
for all $r,s\in R$. Thus $\delta$ is a group homomorphism. It is injective,
because $d_r=d_s$
implies that the $x$-th coordinates agree for every $x\in\Omega$, and hence
$r=s$. Therefore $\delta(R)$ is a subgroup of $G$ isomorphic to $R$. 
By the correspondence between regular subgroups of holomorphs and skew braces
\cite[Theorem~4.2]{GV17}, the skew brace $B=\KT_T(C,\chi)$
has underlying set $A^\Omega$, and its multiplicative group is obtained by
transporting the group structure of $G$ through the bijection $\Pi:G\longrightarrow A^\Omega$ given by $\Pi((r_x)_{x\in\Omega},u)
=
(\pi(r_x))_{x\in\Omega}.$
Applying this to a diagonal element gives $\Pi(d_r)
=
(\pi(r))_{x\in\Omega}.$
Since $\pi:R\to A$ is bijective, it follows that
$\Pi(\delta(R))
=
\{(a)_{x\in\Omega}:a\in A\}
=
\Delta(A).$
We now verify that $\Delta(A)$ is a subbrace of $B$. First, $\Delta(A)$ is a subgroup of $(B,+)=A^\Omega$, since $(a)_{x\in\Omega}+(b)_{x\in\Omega} = (a+b)_{x\in\Omega}$ and $-(a)_{x\in\Omega} = (-a)_{x\in\Omega}.$ It remains to check that $\Delta(A)$ is a subgroup of $(B,\circ)$. Let $a,b\in A$. Since $\pi$ is bijective, there exist unique $r,s\in R$ such that $ a=\pi(r)$ and $b=\pi(s)$. Then $ (a)_{x\in\Omega}=\Pi(d_r)$  and $(b)_{x\in\Omega}=\Pi(d_s).$ Since $\Pi$ identifies $G$ with the multiplicative group $(B,\circ)$, we obtain \begin{align*} (a)_{x\in\Omega}\circ(b)_{x\in\Omega} &= \Pi(d_r)\circ\Pi(d_s)= \Pi(d_rd_s)= \Pi(d_{rs})= (\pi(rs))_{x\in\Omega} \in\Delta(A). \end{align*}Therefore $\Delta(A)$ is a subgroup of both $(B,+)$ and $(B,\circ)$, and hence it is a subbrace of $B$.
Finally, consider the diagonal map $\varphi:C\longrightarrow\Delta(A)$ given by $\varphi(a)=(a)_{x\in\Omega}.$
This map is clearly an isomorphism between the additive groups
$(C,+)=A$ and $\Delta(A)$. We now check that it also preserves the multiplicative operation. Indeed we have
\begin{align*}
\varphi(a)\circ_B\varphi(b)
&=
\Pi(d_r)\circ_B\Pi(d_s)=
\Pi(d_rd_s)=
\Pi(d_{rs})=
(\pi(rs))_{x\in\Omega}.
\end{align*}
On the other hand, since $\pi$ is the identity map on the underlying set
of $C$, we have $a\circ_C b=\pi(rs),$
and therefore $\varphi(a\circ_C b)
=
(\pi(rs))_{x\in\Omega}.$
Thus $\varphi(a\circ_C b)
=
\varphi(a)\circ_B\varphi(b).$
Hence $\varphi$ is an isomorphism of skew braces, and consequently
$\Delta(A)$ is a subbrace of $B$ isomorphic to $C$.
\end{proof}

\begin{theorem}[Infinite towers]\label{thm:tower}
Let $C$ be a finite skew brace, write $(C,+)=A$ and $(C,\circ)=R$, and
suppose that $R$ admits an epimorphism onto a non-trivial finite abelian group
$Q$. Fix $p\in\pi(Q)$. Then there exists a sequence of finite skew braces
$C=C_0,C_1,C_2,\ldots$ such that $(C_m,+)\cong A^{p^m}$ for every
$m\geq0$. Moreover, each $C_m$ embeds diagonally as a subbrace of
$C_{m+1}$, and $(C_m,\circ)$ admits an epimorphism onto $Q$ for every $m$.
If $R$ is solvable, then every $(C_m,\circ)$ is solvable.
\end{theorem}

\begin{proof}
Choose an epimorphism $Q\twoheadrightarrow C_p$ and iterate the
kernel--wreath construction using the quotient $C_p$ at every step.
By Proposition~\ref{prop:quotients}, the quotient $Q$ is preserved, so
the construction can be repeated indefinitely. At each step,
$(C_{m+1},+)\cong(C_m,+)^p$, and hence
$(C_m,+)\cong A^{p^m}$. The diagonal embeddings
$C_m\hookrightarrow C_{m+1}$ are given by
Proposition~\ref{prop:diagonal}. Finally, solvability of the multiplicative
groups, when present initially, is preserved by Proposition~\ref{prop:solvable}.
\end{proof}

\section{Preservation of simplicity}\label{sec:simplicity}

We now prove the main permanence result. Throughout this section let $B=\KT_T(C,\chi)$ and $(B,+)=A^\Omega$
and identify the coordinate factors with groups $A_\omega\cong A$, $\omega\in\Omega$.

\begin{lemma}\label{lem:intersection-ideal}
Assume that $C$ is simple. Let $I$ be an ideal of $B$ and put $J_\omega=I\cap A_\omega.$
Then $J_\omega$, identified with a subgroup of $A$, is an ideal of $C$ for every $\omega\in\Omega$.
\end{lemma}

\begin{proof}
Since $I\trianglelefteq (B,+)=A^\Omega$, we have $J_\omega=I\cap A_\omega\trianglelefteq A_\omega\cong A.$
We first show that $J_\omega$ is invariant under the lambda action of $C$. Fix $r\in R$. Choose $\nu\in\Omega$, with $\nu\neq\omega$, and choose
$s\in R$ such that $\chi(s)=\chi(r)^{-1}.$
Let $c_{r,s}\in R^\Omega$ be defined by $(c_{r,s})_\omega=r,$ $(c_{r,s})_\nu=s$ and $(c_{r,s})_x=1_R$ for $x\notin\{\omega,\nu\}.$
By Lemma~\ref{lem:elements}, $(c_{r,s},1_T)\in G.$ Recall that the lambda action of an element
$g=((r_x)_{x\in\Omega},u)\in G$
on $A^\Omega$ is given by $\lambda^B_g((a_x)_{x\in\Omega})
=
\bigl(\lambda^C_{r_x}(a_{u^{-1}x})\bigr)_{x\in\Omega}.$
Since the top component of $(c_{r,s},1_T)$ is $1_T$, this element does not
permute the coordinates. In particular, it preserves the factor $A_\omega$,
and its restriction to $A_\omega$ is precisely $\lambda^C_r$. Hence $\lambda^C_r(J_\omega)=J_\omega,$
because $I$ is invariant under the lambda action of $B$. Since $r\in R$ was
arbitrary, $J_\omega$ is invariant under the whole lambda action $\lambda^C:R\longrightarrow\Aut(A).$
Since $J_\omega$ is an additive subgroup and is invariant under every
$\lambda^C_r$, it is closed under the multiplicative operation of $C$; as
$C$ is finite, $J_\omega$ is therefore a subbrace of $C$.
It remains to prove that $J_\omega$ is normal in the multiplicative group
of $C$, or equivalently that $\pi^{-1}(J_\omega)\trianglelefteq R.$ Let $a\in J_\omega$ and put $x=\pi^{-1}(a)\in R.$
Let $\iota_\omega(a)\in A^\Omega$
denote the element having entry $a$ in the $\omega$-coordinate and $0$
in every other coordinate. Since $a\in J_\omega$, we have $\iota_\omega(a)\in I.$ Let
$g=\Pi^{-1}(\iota_\omega(a))\in G.$
Since $\Pi((r_x)_{x\in\Omega},u)
=
(\pi(r_x))_{x\in\Omega},$
the base component of $g$ is
\[
b=(b_y)_{y\in\Omega},
\qquad
b_\omega=x,\qquad
b_y=1_R\quad(y\neq\omega).
\]
Thus $g=(b,u)$
for some $u\in T$.
Now fix $r\in R$ and consider the diagonal element $d_r=\bigl((r)_{y\in\Omega},1_T\bigr)\in G.$
Since $I\trianglelefteq(B,\circ)$ and $g$ corresponds to an element of $I$,
we have $d_rgd_r^{-1}\in I.$ Let us compute this conjugate in the wreath product. Since every diagonal
vector is fixed by the permutation action of $T$, ${}^u(r^{-1})_{y\in\Omega}
=
(r^{-1})_{y\in\Omega}.$
Therefore
\begin{align*}
d_rgd_r^{-1}
&=
\bigl((r)_{y\in\Omega},1_T\bigr)
(b,u)
\bigl((r^{-1})_{y\in\Omega},1_T\bigr)=
\left(
(r b_y r^{-1})_{y\in\Omega},
u
\right).
\end{align*}
Since $b_y=1_R$ for $y\neq\omega$, the base component of this conjugate is
trivial outside the $\omega$-coordinate, while its $\omega$-coordinate is
$rxr^{-1}$. Applying $\Pi$, we obtain
\[
\Pi(d_rgd_r^{-1})=\iota_\omega\bigl(\pi(rxr^{-1})\bigr).
\]
Since $d_rgd_r^{-1}\in I$, it follows that $\pi(rxr^{-1})\in J_\omega.$
Thus
\[
r\,\pi^{-1}(J_\omega)\,r^{-1}
\subseteq
\pi^{-1}(J_\omega)
\qquad
\text{for every }r\in R.
\]
Applying the same argument to $r^{-1}$ gives equality, and hence $\pi^{-1}(J_\omega)\trianglelefteq R.$ Therefore $J_\omega$ is normal in $(C,+)$, invariant under the lambda action
of $C$, and normal in $(C,\circ)$. Hence $J_\omega$ is an ideal of $C$.
\end{proof}

\begin{corollary}\label{cor:zero-intersections}
Assume that $C$ is simple. If $I$ is a proper ideal of $B$, then $I\cap A_\omega=0$ with $\omega\in\Omega$.
\end{corollary}

\begin{proof}
By Lemma~\ref{lem:intersection-ideal}, each $I\cap A_\omega$ is either $0$ or $A_\omega$. By Lemma~\ref{lem:elements}(3), the lambda action of $G$ is transitive on the coordinate factors. Hence if one intersection is full, then every coordinate factor lies in $I$, so $I=B$.
\end{proof}

\begin{lemma}\label{lem:projection-ideal}
Assume that $C$ is simple and let $I$ be a proper ideal of $B$. Then $I\leq Z(A)^\Omega.$
Moreover, for each $\omega\in\Omega$, the projection $H_\omega=\pr_\omega(I)\leq A$
is an ideal of $C$.
\end{lemma}

\begin{proof}
Let $z=(z_\gamma)_{\gamma\in\Omega}\in I$ and fix $\omega\in\Omega$.
We first show that the $\omega$-coordinate of every element of $I$
belongs to $Z(A)$. Take $c\in A_\omega$. Since $I\trianglelefteq(B,+)=A^\Omega$, we have
$[z,c]_+\in I$. Moreover, the group operation in $A^\Omega$ is
coordinatewise, so $[z,c]_+$ is trivial outside the $\omega$-coordinate,
while its $\omega$-coordinate is $[z_\omega,c_\omega]_+$. Hence
$[z,c]_+\in I\cap A_\omega$. By
Corollary~\ref{cor:zero-intersections}, $I\cap A_\omega=0$, and therefore
$[z_\omega,c_\omega]_+=0$. Since $c_\omega\in A$ is arbitrary, we obtain
$z_\omega\in Z(A)$. As this holds for every $\omega\in\Omega$ and every
$z\in I$, it follows that $I\leq Z(A)^\Omega.$
Now put $H_\omega=\pr_\omega(I)$. Since $I$ is an additive subgroup of
$A^\Omega$, $H_\omega$ is an additive subgroup of $A$. By the previous
paragraph, $H_\omega\leq Z(A)$, and hence $H_\omega\trianglelefteq A$. We next prove that $H_\omega$ is invariant under the lambda action of $C$.
Let $a\in H_\omega$ and choose $z=(z_\gamma)_{\gamma\in\Omega}\in I$
with $z_\omega=a$. Fix $r\in R$. Choose $\nu\neq\omega$ and
$s\in R$ such that $\chi(s)=\chi(r)^{-1}$. Let
$c_{r,s}\in R^\Omega$ be the element with entries $r$ at $\omega$,
$s$ at $\nu$, and $1_R$ elsewhere. By Lemma~\ref{lem:elements},
$g_{r,s}:=(c_{r,s},1_T)\in G$. Since the top component of $g_{r,s}$ is $1_T$, its lambda action does not
permute the coordinates. In particular,
\[
\pr_\omega\bigl(\lambda^B_{g_{r,s}}(z)\bigr)
=
\lambda^C_r(z_\omega)
=
\lambda^C_r(a).
\]
Since $I$ is lambda-invariant, $\lambda^B_{g_{r,s}}(z)\in I$, and hence
$\lambda^C_r(a)\in H_\omega$. Thus
$\lambda^C_r(H_\omega)\subseteq H_\omega$. Applying the same argument to
$r^{-1}$ gives $\lambda^C_r(H_\omega)=H_\omega
\qquad\text{for every }r\in R.$
It follows that $H_\omega$ is a subbrace of $C$. Equivalently,
$\pi^{-1}(H_\omega)$ is a subgroup of $R$. Indeed, if
$x,y\in\pi^{-1}(H_\omega)$, then by the crossed-homomorphism identity $\pi(xy)=\pi(x)+\lambda^C_x(\pi(y))\in H_\omega.$
Thus $\pi^{-1}(H_\omega)$ is closed under multiplication, and since $R$
is finite, it is a subgroup of $R$.
It remains to prove multiplicative normality. Let $a\in H_\omega$ and
choose $z\in I$ with $z_\omega=a$. Write
$\Pi^{-1}(z)=((x_\gamma)_{\gamma\in\Omega},u)\in G.$
By the definition of $\Pi$, we have $z_\gamma=\pi(x_\gamma)$ for every
$\gamma\in\Omega$, and in particular $x_\omega=\pi^{-1}(a)$.
Fix $r\in R$ and consider the diagonal element
$d_r=((r)_{\gamma\in\Omega},1_T)\in G$. Since
$I\trianglelefteq(B,\circ)$, conjugating $\Pi^{-1}(z)$ by $d_r$ gives
again an element corresponding to $I$. Since diagonal base vectors are
fixed by the permutation action of $T$, we have
\[
d_r\,\Pi^{-1}(z)\,d_r^{-1}
=
\bigl((rx_\gamma r^{-1})_{\gamma\in\Omega},u\bigr).
\]
Applying $\Pi$, we obtain an element of $I$ whose $\omega$-coordinate is
$\pi(rx_\omega r^{-1})$. Hence
$\pi(rx_\omega r^{-1})\in H_\omega.$
Since $x_\omega=\pi^{-1}(a)$ and $a\in H_\omega$ was arbitrary, we get
$r\pi^{-1}(H_\omega)r^{-1}\subseteq\pi^{-1}(H_\omega)$ for every
$r\in R$. Applying the same argument to $r^{-1}$ yields equality, so $\pi^{-1}(H_\omega)\trianglelefteq R.$
Therefore $H_\omega$ is normal in $(C,+)$, invariant under the lambda
action of $C$, and normal in $(C,\circ)$. Hence $H_\omega$ is an ideal
of $C$.
\end{proof}

\begin{theorem}[Kernel--wreath permanence theorem]\label{thm:main}
Let $C$ be a finite simple skew brace which is not a trivial skew brace. Let $\chi:(C,\circ)\twoheadrightarrow T$
be an epimorphism onto a non-trivial finite abelian group. Then $\KT_T(C,\chi)$
is simple.
\end{theorem}

\begin{proof}
Write $(C,+)=A$ and $(C,\circ)=R$, and put $B=\KT_T(C,\chi)$. Let $I$ be a proper ideal of $B$. We prove that $I=0$. Suppose, for a contradiction, that $I\ne0$. Then $H_\omega=\pr_\omega(I)$ is non-zero for some $\omega$. By Lemma~\ref{lem:projection-ideal}, $H_\omega$ is an ideal of the simple skew brace $C$, so $H_\omega=A$. The same lemma gives $H_\omega\leq Z(A)$; hence $A=Z(A)$ and $C$ is a brace. Put $K=\ker\chi$. Let $k\in K$. The base element having entry $k$ in the $\omega$-coordinate and identity elsewhere belongs to $G$. Its lambda-map acts as $\lambda^C_k$ on $A_\omega$ and trivially on all other coordinates. Since $H_\omega=A$, for every $a\in A$ there exists $z\in I$ with $z_\omega=a$. The difference between $z$ and its image under this lambda-map belongs to $I$ and is supported only at $\omega$. Corollary~\ref{cor:zero-intersections} therefore gives $\lambda^C_k(a)=a$ for every $a\in A$.
Thus $K\leq\ker\lambda^C$. Since $A$ is abelian,
$\Soc(C)=\ker\lambda^C$. The socle is an ideal of $C$, and it cannot be
all of $C$, since otherwise $C$ would be trivial. Hence the simplicity of $C$
gives $\Soc(C)=0$, and therefore $K=1$.
Therefore $R\cong T$ is abelian, and hence nilpotent. By
\cite[Corollary~3.2]{DameleErcan26}, the finite simple skew brace $C$ is
isomorphic to $\operatorname{Triv}(C_p)$ for some prime $p$, contrary to the
hypothesis. Therefore $I=0$, and $B$ is simple.
\end{proof}

\begin{corollary}[Simple iteration]\label{cor:simple-iteration}
Under the hypotheses of Theorem~\ref{thm:iteration}, assume in addition that
$C$ is simple and is not a trivial skew brace. Then the skew braces $B_n$ may
be chosen to be simple.
\end{corollary}

\begin{proof}
Use the sequence $C_0,C_1,\ldots,C_m$ constructed in the proof of
Theorem~\ref{thm:iteration}. We prove inductively that every $C_i$ is simple
and is not a trivial skew brace. This holds for $C_0=C$. If it holds for
$C_{i-1}$, then Theorem~\ref{thm:main} shows that
$C_i=\KT_{C_{p_i}}(C_{i-1},\chi_i)$ is simple. Moreover,
$(C_i,+)\cong(C_{i-1},+)^{p_i}$ with $p_i\geq2$. Hence $(C_i,+)$ has a
non-zero proper normal direct factor. If $C_i$ were trivial, this factor would
be an ideal of $C_i$, contradicting simplicity. Thus $C_i$ is not trivial,
and the induction continues. In particular, $B_n=C_m$ is simple.
\end{proof}

\begin{corollary}[Simple infinite towers]\label{cor:simple-tower}
Under the hypotheses of Theorem~\ref{thm:tower}, assume in addition that
$C$ is simple and is not a trivial skew brace. Then the sequence in
Theorem~\ref{thm:tower} may be chosen so that every $C_m$ is simple. If
$(C,\circ)$ is solvable, then every $(C_m,\circ)$ is solvable.
\end{corollary}

\begin{proof}
Apply the same induction as in Corollary~\ref{cor:simple-iteration} to the
sequence constructed in the proof of Theorem~\ref{thm:tower}. Solvability is
already part of Theorem~\ref{thm:tower}.
\end{proof}

\begin{corollary}\label{cor:solvable-seed}
Let $C$ be a finite simple skew brace which is not a trivial skew brace. If
$(C,\circ)$ is solvable, then $C$ is the first term of an infinite tower of
finite simple skew braces with solvable multiplicative groups.
\end{corollary}

\begin{proof}
The finite solvable group $(C,\circ)$ is non-trivial and therefore is not
perfect. Hence $(C,\circ)/(C,\circ)'\neq1$. Choose a prime divisor $p$ of
$|(C,\circ)/(C,\circ)'|$ and apply Corollary~\ref{cor:simple-tower} with
$Q=(C,\circ)/(C,\circ)'.$
\end{proof}

\section{Applications}\label{sec:applications}

\subsection{Non-abelian simple additive groups}

Let $S$ be a finite non-abelian simple group and let
$R\leq\Hol(S)$ be a solvable regular subgroup. The corresponding skew
brace $C$ has additive group $(C,+)\cong S$ and multiplicative group
$(C,\circ)\cong R$. Since every ideal of $C$ is, in particular, a normal
subgroup of $(C,+)$, the simplicity of $S$ implies that $C$ is simple.
Moreover, $C$ is not a trivial skew brace, since $S$ is non-solvable
whereas $R$ is solvable. Since $R$ is a non-trivial finite solvable group, its abelianization
$R/R'$ is non-trivial. Thus the kernel--wreath construction produces an
infinite family for every prime divisor of $|R/R'|$.

\begin{corollary}\label{cor:simple-additive}
Let $S$ be a finite non-abelian simple group and let
$R\leq\Hol(S)$ be a solvable regular subgroup. Then, for every prime
$p\in\pi(R/R')$, there exists an infinite sequence of finite simple skew
braces $B_0,\ B_1,\ B_2,\ldots$
such that $(B_m,+)\cong S^{p^m}$
for every $m\geq0$, and every $(B_m,\circ)$ is solvable. Moreover, the
braces may be chosen so that $B_m$ embeds diagonally as a subbrace of
$B_{m+1}$ for every $m\geq0$.
\end{corollary}

\begin{proof}
Let $C$ be the skew brace associated with the regular subgroup $R$. As
observed above, $C$ is simple and is not a trivial skew brace. If
$p\in\pi(R/R')$, then the finite abelian group $R/R'$ admits an
epimorphism onto $C_p$. Hence $R\twoheadrightarrow C_p$, and the result follows from Corollary~\ref{cor:simple-tower}.
\end{proof}

Tsang proved that the finite non-abelian simple groups whose holomorph
contains a solvable regular subgroup are precisely
$\PSL_2(q)$, where $q\neq2,3$ is a prime power, together with
\[
\PSL_3(3),\quad
\PSL_3(4),\quad
\PSL_3(8),\quad
\PSU_3(8),\quad
\PSU_4(2),\quad
M_{11};
\]
see \cite[Theorem~1.3]{Tsang23}. Thus each solvable regular subgroup
arising from these groups yields, for every prime occurring in its
abelianization, an infinite family as in
Corollary~\ref{cor:simple-additive}. We now give two explicit examples.

\subsubsection{The case \texorpdfstring{$A_5$}{A5}}

Consider $S=A_5$. The group $A_5$ admits the exact factorization
$A_5=A_4C_5$, where $A_4\cap C_5=1$. By
\cite[Proposition~2.4]{Tsang23}, this gives a solvable regular subgroup $R\cong A_4\times C_5$
of $\Hol(A_5)$. Hence there exists a simple skew brace $C$ with
$(C,+)\cong A_5$ and $(C,\circ)\cong A_4\times C_5$. Since $A_4'=V_4$, we have $R/R'
\cong
(A_4/V_4)\times C_5
\cong
C_3\times C_5
\cong C_{15}.$
In particular, both $3$ and $5$ belong to $\pi(R/R')$.
Corollary~\ref{cor:simple-additive} therefore gives two infinite towers
of finite simple skew braces:
\[
(B_m^{(3)},+)\cong A_5^{\,3^m},
\qquad
(B_m^{(5)},+)\cong A_5^{\,5^m},
\]
for every $m\geq0$, with solvable multiplicative groups in both cases. More generally, Corollary~\ref{cor:simple-iteration} shows that for every
$n\geq2$ satisfying $\pi(n)\subseteq\{3,5\}$ there exists a finite
simple skew brace $B_n$ such that
$(B_n,+)\cong A_5^n$ and $(B_n,\circ)$ is solvable.

\subsubsection{The case \texorpdfstring{$\PSU_3(8)$}{PSU3(8)}}

We next consider $S=\PSU_3(8)$. In the proof of
\cite[Theorem~1.3, case~(3)]{Tsang23}, Tsang constructs a solvable
regular subgroup $R\leq\Hol(\PSU_3(8))$ which admits a quotient $H$ of
order $|H|=513=3^3\cdot19.$
The Sylow $19$-subgroup $P$ of $H$ is normal, since its number divides
$27$ and is congruent to $1$ modulo $19$. Thus $H/P$ is a group of
order $27$, and hence admits a quotient isomorphic to $C_3$. Consequently,
$R\twoheadrightarrow C_3$, so in particular $3\in\pi(R/R')$. Applying Corollary~\ref{cor:simple-additive}, for every $m\geq0$ there
exists a finite simple skew brace $B_m$ such that
\[
(B_m,+)\cong\PSU_3(8)^{\,3^m},
\]
and $(B_m,\circ)$ is solvable. 

\subsection{The simple skew braces of order 12}

Let $C$ be one of the two simple skew braces of order $12$ with
$(C,+)\cong A_4$ and
$(C,\circ)\cong C_3\rtimes C_4$, where the action of $C_4$ on $C_3$ is
non-trivial; see \cite{KSV21,Byott26}. The multiplicative group admits a
quotient isomorphic to $C_2$. Hence Corollary~\ref{cor:simple-tower} gives the following infinite
family.

\begin{corollary}\label{cor:A4}
For every $m\geq0$ there exists a finite simple skew brace $B_m$ such
that $(B_m,+)\cong A_4^{\,2^m}$
and $(B_m,\circ)$ is solvable.
\end{corollary}

\subsection{Byott's family}

Let $p$ and $q$ be primes such that
$q\mid (p^p-1)/(p-1)$, and let $C$ be one of Byott's simple skew braces
of order $p^pq$ \cite{Byott26}. Write
$(C,+)\cong N=V\rtimes C_q$, where $V\cong C_p^p$, and
$(C,\circ)\cong R=C_q\rtimes P$, where $P$ is a $p$-group of order
$p^p$. In Byott's presentation there are elements $X,Z\in R$ satisfying
$X^q=1$ and $ZX=X^pZ$. Hence
$[Z,X]=X^{p-1}$. The multiplicative order of $p$ modulo $q$ is $p$, so
$p\not\equiv1\pmod q$. It follows that $X^{p-1}$ generates
$\langle X\rangle$, and therefore
$\langle X\rangle\leq R'$.

Now $R/\langle X\rangle$ is a non-trivial finite $p$-group, and hence
admits a quotient isomorphic to $C_p$. Consequently,
$R\twoheadrightarrow C_p$. Applying Corollary~\ref{cor:simple-tower}
gives the following.

\begin{corollary}\label{cor:byott}
For every $m\geq0$ there exists a finite simple skew brace $B_m$ such
that $(B_m,+)\cong N^{p^m}$
and $(B_m,\circ)$ is solvable. Moreover, the braces may be chosen as a
nested tower under diagonal subbrace embeddings.
\end{corollary}

\end{document}